\documentclass[12pt]{amsart}

\usepackage{latexsym,amsmath,amssymb,amsfonts,amsthm,mathrsfs}
\usepackage{graphicx}
\usepackage{cite}
\usepackage{enumitem}
\usepackage{enumerate}
\usepackage{epstopdf}
\usepackage{epsfig}
\usepackage{subfigure}
\usepackage{tikz}
\usepackage{caption}
\usepackage{hyperref}

\newtheorem{thm}{Theorem}[section]
\newtheorem{lem}{Lemma}[section]
\newtheorem{cor}{Corollary}[section]
\newtheorem{prop}{Proposition}[section]

\theoremstyle{definition}

\newtheorem{remark}{Remark}[section]

\numberwithin{equation}{section}
\newcommand{\RNum}[1]{\uppercase\expandafter{\romannumeral #1\relax}}

\begin{document}

\title[Variations of colored multiset Eulerian polynomials and applications]
{Variations of colored multiset Eulerian polynomials and applications}

\author{Xue~Yan}
\address{School of Mathematical Sciences\\
Qufu Normal University\\
Qufu 273165, PR China}
\email{yanxueyx@126.com}

\date{\today}
\thanks{ \textit{Mathematics Subject
Classifications}: 52B20, 05A05, 05A15, 26C10}
\thanks{ \textit{Key words and phrases}.
Colored multiset permutation, Eulerian polynomial, real-rootedness, interlacing, Ehrhart polynomial, $h^\ast$-polynomial, half-open lattice polytopes.
}

\begin{abstract}
Deligeorgaki, Han and Solus introduced colored multiset Eulerian polynomials,
derived a generating function identity which generalizes MacMahon's identity,
proved their self-interlacing under suitable parameter conditions,
and identified these polynomials as the \(h^*\)-polynomials of direct products of dilated lattice simplices.
In this paper,
we introduce an ascent analogue of the colored multiset Eulerian polynomial,
derive an explicit generating function identity for this polynomial,
show that it is equal to the $h^*$-polynomial of a family of half-open lattice polytopes,
and verify that this ascent polynomial also satisfies self-interlacing under the same parameter conditions.
By establishing recurrence relations,
we prove that both polynomials are real-rooted for all positive integer parameters.
The obtained identities are further applied to interpret combinatorially the $h^*$-polynomials of Pitman--Stanley polytopes,
composition polytopes and a family of reflexive lattice polytopes defined from preorders.
\end{abstract}

\maketitle

\section{Introduction}
\hspace*{\parindent}
Let $\mathfrak{S}_d$ be the symmetric group on the set $[d]= \{1,2,\dots,d\}$.
An element $\pi$ of $\mathfrak{S}_d$ is represented by the word $\pi_1\pi_2\cdots\pi_d $,
with $\pi(i)=\pi_i$ for $i\in [d]$.
Let
\begin{align*}
\operatorname{asc}(\pi) &= \#\{i\in[d-1]:\pi_i <\pi_{i+1}\},\\
\operatorname{des}(\pi) &= \#\{i\in[d-1]:\pi_i >\pi_{i+1}\}
\end{align*}
denote the number of \emph{ascents} and \emph{descents} of $\pi$,
respectively.
It is well known that the statistics $\operatorname{asc}(\pi)$ and $\operatorname{des}(\pi)$ are equidistributed over $\mathfrak{S}_d$.
Their common generating function is the \emph{Eulerian polynomial}~\cite{FS70,Pet15},
defined by
\[
A_d(x)=\sum_{\pi\in\mathfrak{S}_d}x^{\operatorname{asc}(\pi)}
=\sum_{\pi\in\mathfrak{S}_d}x^{\operatorname{des}(\pi)}.
\]
It satisfies the recurrence relation
\[
A_d(x)=\big(1+(d-1)x\big)A_{d-1}(x)+x(1-x)A_{d-1}'(x)
\]
with the initial condition $A_0(x)=1$ (see Comtet~\cite[Exercise VII. 3]{Com74}, for instance).
The Eulerian polynomial has only real roots for all integers $d\ge 1$~\cite{Fro10}.
Hence,
its coefficients are unimodal and log-concave.
These properties have been extended in many directions,
including colored permutations~\cite{Ste92, Ste94},
multiset permutations~\cite{Mac04,Sim84},
signed multiset permutations~\cite{Lin15},
and various combinations thereof.

Recently, Deligeorgaki, Han and Solus \cite{DHS26} studied colored multiset permutations and introduced \emph{the colored multiset Eulerian polynomial}
\[
A_{\mathbf{m}}^{\mathbf{r}}(x)=\sum_{\pi^{\mathbf{c}}\in\mathfrak{S}_{\mathbf{m}}^{\mathbf{r}}}x^{\operatorname{des}(\pi^{\mathbf{c}})}.
\]
Here \(\mathbf{m}=(m_1,\dots,m_d)\in\mathbb{Z}_{>0}^d\) records the multiplicities of the letters \(1,\dots,d\),
with $m=m_1+\cdots+m_d$,
and \(\mathbf{r}=(r_1,\dots,r_d)\in\mathbb{Z}_{>0}^d\) records the number of available colors for each letter. The set \(\mathfrak{S}_{\mathbf{m}}^{\mathbf{r}}\) consists of words \(\pi^{\mathbf{c}}=\pi_1^{c_1}\cdots\pi_m^{c_m}\) in which each \(k\in [d]\) appears \(m_k\) times and each entry carries a color \(c_i\in[r_{\pi_i}]\).
We append the sentinel element $(d+1)^1$ to the end of each such word,
i.e., $\pi_{m+1}^{c_{m+1}}=(d+1)^1.$
Comparisons are taken with respect to the \emph{color order}~\cite{Ste94}:
\[
1^1<2^1<\cdots<d^1<(d+1)^1
<1^2<2^2<\cdots<d^2<1^3<\cdots.
\]
For any index $j\in[m]$,
we say that $j$ is a \emph{descent} of $\pi^{\mathbf{c}}$ if
\(
\pi_j^{c_j}>\pi_{j+1}^{c_{j+1}}.
\)
We denote by $\operatorname{des}(\pi^{\mathbf{c}})$ the number of descents of $\pi^{\mathbf{c}}$.
For a different descent statistic on colored multiset permutations,
see~\cite[Definition~5.1]{Tie26}.

The colored multiset Eulerian polynomial $A_{\mathbf{m}}^{\mathbf{r}}(x)$ generalizes the classical Eulerian polynomials
(when $\mathbf m=\mathbf r=\mathbf 1$),
MacMahon's multiset Eulerian polynomials
(when $\mathbf r=\mathbf 1$),
Lin's signed multiset Eulerian polynomials
(when $\mathbf r=2\mathbf 1$),
and Steingr\'{\i}msson's colored Eulerian polynomials
(when $\mathbf m=\mathbf 1$ and $\mathbf r=r\mathbf 1$)
(see, for instance~\cite{DHS26, Lin15, Mac04, Ste94}).

Moreover, Deligeorgaki, Han and Solus established the generating function identity~\cite[Corollary 2.2]{DHS26}
\begin{equation}\label{A1}
\frac{A_{\mathbf{m}}^{\mathbf{r}}(x)}{(1-x)^{m+1}}
= \sum_{n\ge 0}\prod_{i=1}^d\binom{r_i n+m_i}{m_i}x^n,
\end{equation}
which recovers MacMahon's identity and its signed analogue due to
Lin~\cite[Theorem~6]{Lin15}.
They further proved that, under the
condition \(r_j\ge m_j+1\) for all \(j\), the polynomials
\(A_{\mathbf m}^{\mathbf r}(x)\) are self-interlacing (see Section 2 for undefined terminology).
Consequently,
they are real-rooted, log-concave, unimodal, alternatingly increasing,
and bi-\(\gamma\)-positive.
In addition, they showed that
\begin{equation}\label{Ah*}
A_{\mathbf{m}}^{\mathbf{r}}(x) = h^*\bigl(P_{\mathbf{m}}^{\mathbf{r}};x\bigr),
\end{equation}
where $h^*(P;x)$ denotes the $h^*$-polynomial of a polytope $P$;
see Section~2 for its definition.
The polytope
\[
P_{\mathbf{m}}^{\mathbf{r}}:=\prod_{j=1}^d r_j\Delta_{m_j}\subset\mathbb{R}^m,\quad m=\sum_{j=1}^d m_j
\]
is the direct product of dilated standard simplices.
For every positive integer $s$, let
\[
\Delta_s
=
\left\{
\mathbf{x}\in\mathbb{R}_{\ge0}^{s}
:
\sum_{j=1}^{s}x_j\le1
\right\}
\]
denote the $s$-dimensional closed standard simplex.

Motivated by their work,
we investigate the corresponding ascent statistic.
We prepend the sentinel element $1^1$ to each such word,
i.e., $\pi_0^{c_0}=1^1$.
For any index $j\in[0,m-1]=\{0,1,\dots,m-1\}$,
we say that $j$ is an \emph{ascent} of $\pi^{\mathbf{c}}$ if
\(
\pi_j^{c_j}<\pi_{j+1}^{c_{j+1}}.
\)
We define
$\operatorname{ASC}(\pi^{\mathbf{c}})$ and
$\operatorname{asc}(\pi^{\mathbf{c}})$
as the set and the number of ascents of $\pi^{\mathbf{c}}$,
respectively.
In this paper,
we consider the corresponding colored multiset Eulerian polynomial defined by
\[
S_{\mathbf{m}}^{\mathbf{r}}(x)
=
\sum_{\pi^{\mathbf{c}}\in\mathfrak{S}_{\mathbf{m}}^{\mathbf{r}}}
x^{\operatorname{asc}(\pi^{\mathbf{c}})}.
\]
In contrast to ordinary (colored) permutations,
ascents and descents are no longer equidistributed under this model with the leading sentinel element.
Therefore,
the ascent generating polynomial $S_{\mathbf{m}}^{\mathbf{r}}(x)$ does not coincide with $A_{\mathbf{m}}^{\mathbf{r}}(x)$.

Our first main result establishes a generating function identity for $S_{\mathbf{m}}^{\mathbf{r}}(x)$.
\begin{thm}\label{thm:Sgf}
For all $\mathbf{m},\mathbf{r}\in\mathbb{Z}_{>0}^d$,
\begin{equation}\label{eq:Sgf}
\frac{S_{\mathbf{m}}^{\mathbf{r}}(x)}{(1-x)^{m+1}}
= \sum_{n \ge 0} \binom{r_1 n + m_1}{m_1} \prod_{i=2}^d \binom{r_i n + m_i - 1}{m_i} x^n.
\end{equation}
\end{thm}

This identity naturally leads to a geometric interpretation of the ascent polynomials as an $h^*$-polynomial.
\begin{thm}\label{Sgeo}
For all $\mathbf{m},\mathbf{r}\in\mathbb{Z}_{>0}^d$, define the polytope
\[
Q_{\mathbf{m}}^{\mathbf{r}}=r_1\Delta_{m_1}\times\prod_{i=2}^d r_i\widetilde{\Delta}_{m_i}\subset\mathbb{R}^m.
\]
Then
\[
S_{\mathbf{m}}^{\mathbf{r}}(x)=h^*\bigl(Q_{\mathbf{m}}^{\mathbf{r}};x\bigr),
\]
where $\Delta_s$ is the $s$-dimensional closed standard simplex
defined above, and
\[
\widetilde{\Delta}_s
=
\left\{
\mathbf{x}\in\mathbb{R}_{\ge0}^{s}
:
\sum_{j=1}^{s}x_j\le1,\quad x_s>0
\right\}
\]
is the $s$-dimensional half-open standard simplex.
\end{thm}

The identity~\eqref{eq:Sgf} also enables us to establish
the following self-interlacing property.
Here $\preceq$ denotes the interlacing relation and
$\mathcal{I}_m$ denotes the reciprocal operator with respect to
degree $m$; see Section~2 for the precise definitions.

\begin{thm}\label{thm:Sself}
Suppose that $\mathbf{m},\mathbf{r}\in\mathbb{Z}_{>0}^d$ satisfy $r_j\ge m_j+1$ for all $j\in [d]$.
Then,
\[
\mathcal{I}_m
\bigl(S_{\mathbf{m}}^{\mathbf{r}}\bigr)
\preceq
S_{\mathbf{m}}^{\mathbf{r}}.
\]
\end{thm}This theorem implies several other distributional properties,
including log-concavity, unimodality,
the alternatingly increasing property and bi-\(\gamma\)-positivity~\cite{BS21, DBZ24}.

Finally,
let $\mathbf{m}_{+j}$ denote the vector obtained by increasing the $j$-th component of $\mathbf{m}$ by one.
We obtain the following results.
\begin{thm}\label{rS}
For all $\mathbf{m},\mathbf{r}\in\mathbb{Z}_{>0}^d$, the polynomial $S_{\mathbf{m}}^{\mathbf{r}}$ is real-rooted.
In particular, it is unimodal and log-concave.
Moreover, for all $j\in[d]$,
\[
S_{\mathbf{m}}^{\mathbf{r}}
\preceq
S_{\mathbf{m}_{+j}}^{\mathbf{r}}.
\]
\end{thm}
\begin{thm}\label{rA}
For all $\mathbf{m},\mathbf{r}\in\mathbb{Z}_{>0}^d$, the polynomial $A_{\mathbf{m}}^{\mathbf{r}}$ is real-rooted.
In particular, it is unimodal and log-concave.
Moreover, for all $j\in[d]$,
\[
A_{\mathbf{m}}^{\mathbf{r}}
\preceq
A_{\mathbf{m}_{+j}}^{\mathbf{r}}.
\]
\end{thm}

The proof of Theorem~\ref{thm:Sgf} uses a standard
barred permutation argument; see Gessel and Stanley~\cite{GS78}.
The proof of Theorem~\ref{thm:Sself} follows closely the argument
in~\cite[Theorem~3.11]{DHS26}.
The proofs of Theorems~\ref{rS} and~\ref{rA} are based on new
recurrence relations for
$S_{\mathbf{m}}^{\mathbf{r}}$ and
$A_{\mathbf{m}}^{\mathbf{r}}$, respectively.

The paper is organized as follows.
Section~\ref{2} explains basic definitions and known results about interlacing polynomials,
real-rooted polynomials and the Ehrhart theory of lattice polytopes.
Section~\ref{3} proves the generating function identity for
$S_{\mathbf{m}}^{\mathbf{r}}$,
shows that $S_{\mathbf{m}}^{\mathbf{r}}$ is the $h^*$-polynomial
of the half-open polytope $Q_{\mathbf{m}}^{\mathbf{r}}$ and proves that $S_{\mathbf{m}}^{\mathbf{r}}$ is self-interlacing
when $r_j\ge m_j+1$ for every $j\in[d]$.
Section~\ref{4} finds recurrence relations for $A_{\mathbf{m}}^{\mathbf{r}}$ and $S_{\mathbf{m}}^{\mathbf{r}}$ based on multiplicity vectors
and proves that both polynomials have only real roots for all positive integer vectors $\mathbf{m},\mathbf{r}$.
Section~\ref{5} uses our generating function identities to interpret combinatorially the $h^*$-polynomials of Pitman--Stanley polytopes,
composition polytopes and a family of reflexive lattice polytopes defined from preorders.
\section{Preliminaries}\label{2}
\hspace*{\parindent}
Let $p(x)=c_0+c_1x+\cdots+c_dx^d$ be a polynomial with real coefficients.
We say that $p(x)$ is \emph{unimodal} if its coefficients satisfy
$c_0\leq c_1\leq\cdots\leq c_\ell
\geq\cdots\geq c_{d-1}\geq c_d$
($\ell$ is called a \emph{mode} of the sequence).
It is called \emph{symmetric}
if $c_k=c_{d-k}$ for all $0\leq k\leq d.$
It is \emph{$\gamma$-positive} if $p(x)=\sum_{k=0}^{\lfloor d/2\rfloor} \gamma_k x^k(1+x)^{d-2k}$ for $\gamma_k\ge 0$, with $0\le k\le{\lfloor d/2\rfloor}$.
It is \emph{real-rooted} if either $p(x)\equiv 0$ or every root of $p(x)$ is real.
Every $\gamma$-positive polynomial is symmetric and unimodal, and every real-rooted and symmetric polynomial with nonnegative coefficients is $\gamma$-positive.
We refer the reader to \cite{Ath18,Bra15,Gal05,Pet15,Sta89} for more information.

A polynomial $p(x)=c_0+c_1x+\cdots+c_dx^d$ is called \emph{alternatingly increasing} if
\[
c_0 \le c_d \le c_1 \le c_{d-1} \le \cdots \le c_{\lfloor (d+1)/2\rfloor}.
\]
This property,
which is stronger than unimodality,
has recently been the focus of a variety of conjectures in combinatorics,
as well as a tool to prove unimodality.
The alternatingly
increasing property of $p$ is inherently tied to a unique symmetric decomposition of $p$:
Given a polynomial $p(x)$ of degree at most $d$, its $\mathcal{I}_d$-decomposition is the unique representation
\[
p(x)=a(x)+x\,b(x),
\]
where $a$ is symmetric with respect to degree $d$,
$b$ is symmetric with respect to degree $d-1$, $\deg a\le d$ and $\deg b\le d-1$.
A basic observation is that $p$ is alternatingly increasing if and only if both $a$ and $b$ have nonnegative coefficients and are unimodal.
When $a$ and $b$ are both $\gamma$-positive, $p$ is said to be \emph{bi-$\gamma$-positive}.

Let $p(x), q(x)$ be real-rooted polynomials.
Suppose that $(r_i)_{1\le i\le d}$ and
$(s_i)_{1\le i\le \ell}$ are all roots of $p(x)$ and $q(x)$
in nonincreasing order, respectively.
We say that $q(x)$ interlaces $p(x)$,
denoted by $q(x)\preceq p(x)$,
if $\ell=d-1$ or $\ell=d$, and
\[
(s_d\leq)~r_d\leq s_{d-1}\leq r_{d-1}
\leq\cdots\leq r_2\leq s_1\leq r_1.
\]
For notational convenience, let $a\preceq bx+c$ for any constants $a,b,c$ and $p(x)\preceq 0,\ 0\preceq p(x)$ for any real polynomial $p(x)$ with only real roots.
For a polynomial $p$ of degree at most $d$, its reciprocal
with respect to degree $d$ is
\(
\mathcal{I}_d(p)=x^d p(1/x).
\)
A polynomial $p$ is \emph{self-interlacing} with respect to $d$ if $\mathcal{I}_d(p)\preceq p$.
Self-interlacing polynomials are a natural generalization of symmetric, real-rooted polynomials.

Liu and Wang~\cite{LW07} established the following  extremely useful tool for proving the real-rootedness and interlacing property.
\begin{thm}[{[\citen{LW07}]}, Theorem 2.1]\label{LW07}
Let $F$, $p$, $q$ be three real polynomials satisfying the following conditions.
\begin{itemize}
    \item[\rm{(a)}] $F(x) = a(x)p(x) + b(x)q(x)$, where $a(x), b(x)$ are real polynomials and $\deg F = \deg p$ or $\deg p + 1$.
    \item[\rm{(b)}] $p, q$ are real-rooted  and $q \preceq p$.
    \item[\rm{(c)}] $F$ and $q$ have leading coefficients of the same sign.
\end{itemize}
Suppose that $b(r) \le 0$ whenever $p(r) = 0$. Then,
 $F$ is real-rooted and $p \preceq F$.
\end{thm}
We refer the reader to Beck and Robins~\cite{BR07} for background in Ehrhart theory.
For an $s$-dimensional lattice polytope $P\subset\mathbb{R}^d$,
its \emph{Ehrhart polynomial} is the unique polynomial \( \mathcal{L}(P;x)\) with the property that $\mathcal{L}(P;n)=|nP\cap\mathbb{Z}^d|$  for every \( n\in\mathbb{N} \).
The $h^*$-polynomial $h^*(P;x)$ of $P$ is defined by the equation
\[
\sum_{n\ge 0}\mathcal{L}(P;n)x^n=\frac{h^*(P;x)}{(1-x)^{s+1}}.
\]
For the $s$-dimensional standard simplex
\[
\Delta_{s} = \left\{ \mathbf{x}\in\mathbb{R}_{\ge 0}^{s} : \sum_{j=1}^{s} x_j \le 1 \right\}
\]
we have $\mathcal{L}(\Delta_s;x)=\binom{x+s}{s}$ \cite[Theorem 2.2]{BR07}, hence $\mathcal{L}(r\Delta_s;x)=\binom{rx+s}{s}$.

The above notions extend naturally to half-open polytopes.
Let
\[
P=\{x\in\mathbb{R}^s:Ax\le b\}
\]
be a polytope.
A half-open polytope is obtained from $P$ by replacing some of the defining weak inequalities with the corresponding strict inequalities.
Such objects have appeared in the literature; see, for example, Nan Li's work on half-open hypersimplices~\cite{Li12}.
For half-open lattice polytopes,
we use the same definitions of the Ehrhart polynomial and
the $h^*$-polynomial as for lattice polytopes.
As a basic example,
consider the half-open simplex
\[
\widetilde{\Delta}_s=
\left\{
x\in\mathbb R_{\ge0}^s:
\sum_{j=1}^s x_j\le1,\,
x_s>0
\right\}.
\]
Then,
\[
\mathcal{L}(r\widetilde{\Delta}_s;x)=\binom{rx+s-1}{s}.
\]
Indeed,
the lattice points in $r\widetilde{\Delta}_s$ are those in $r\Delta_s$ with positive last coordinate.
Excluding the face $x_s=0$, which is a $(s-1)$-simplex with Ehrhart polynomial $\binom{rx+s-1}{s-1}$,
we get
\[
\mathcal{L}(r\widetilde{\Delta}_s;x)=\binom{rx+s}{s}-\binom{rx+s-1}{s-1}=\binom{rx+s-1}{s}.
\]
For a product of polytopes, the Ehrhart polynomial factors
\[
\mathcal{L}(P_1\times\cdots\times P_k;x)=\prod_{i=1}^k\mathcal{L}(P_i;x).
\]

The subdivision operator $\mathcal{E}:\mathbb{R}[x]\to\mathbb{R}[x]$ is the linear transformation defined by $$\mathcal{E}\binom{x}{k}=x^k$$ for all $k\ge 0$. 
The \emph{magic basis} for polynomials of degree at most $d$ is $\{x^i(1+x)^{d-i}:0\le i\le d\}$.
A polynomial is called \emph{magic positive} if its coefficients in this basis are nonnegative.
In what follows,
we let \( E_k^d := \mathcal{E}(x^k(x+1)^{d-k}), 0 \leq k \leq d \),
and we consider the sets
\[\mathcal{E}_d := \left\{ p = \sum_{k=0}^d c_k E_k^d : c_k \geq 0 \right\}\quad\text{and}\quad \mathcal{A}_d := \{ p \in \mathcal{E}_d : \mathcal{R}_d(p) \preceq p \},\]
where \(\mathcal{R}_d(p)=(-1)^dp(-1-x)\).
The following lemma will be used.

\begin{lem}[{[\citen{BS21}]}, Lemma 2.12]\label{BS21}
If $p\in\mathcal{A}_d$, $q\in\mathcal{E}_d$ and $p\preceq q$, then $q\in\mathcal{A}_d$.
\end{lem}
\section{Main results}\label{3}
\hspace*{\parindent}
This section proves Theorems \ref{thm:Sgf}, \ref{Sgeo} and \ref{thm:Sself}.
The generating function identity~(\ref{eq:Sgf}) in Theorem \ref{thm:Sgf} is essential for proving the other two statements.

\begin{proof}[Proof of Theorem~\ref{thm:Sgf}.]
We prove the identity via a barred permutation combinatorial argument.
Take an arbitrary colored multiset permutation
\(
\pi^{\mathbf{c}}
=
\pi_1^{c_1}\cdots\pi_m^{c_m}
\in\mathfrak{S}_{\mathbf{m}}^{\mathbf{r}},
\)
which we extend by prepending the leading sentinel
\(\pi_0^{c_0}=1^1\).
The resulting word has \(m+1\) gaps: one immediately after the
sentinel, one between every pair of adjacent entries of
\(\pi^{\mathbf c}\) and one after the final entry
\(\pi_m^{c_m}\).

We place vertical bars into these gaps according to the following
rule: every gap indexed by an ascent
\(i\in\operatorname{ASC}(\pi^{\mathbf{c}})\)
contains at least one bar, while all other gaps may contain any
nonnegative number of bars.
Let \(\mathcal{B}(\pi^{\mathbf{c}})\) denote the collection of all
barred permutations constructed from \(\pi^{\mathbf{c}}\) under this
rule and define
\[
\mathcal{B}
=
\bigcup_{\pi^{\mathbf{c}}
\in\mathfrak{S}_{\mathbf{m}}^{\mathbf{r}}}
\mathcal{B}(\pi^{\mathbf{c}}).
\]

Given any barred permutation
\(\theta\in\mathcal{B}(\pi^{\mathbf{c}})\),
let \(\nu_0,\ldots,\nu_m\) be the numbers of bars placed in the
\(m+1\) gaps, ordered from left to right.
Assign to \(\theta\) the monomial weight
\[
\operatorname{wt}(\theta)
=
x^{\nu_0+\nu_1+\cdots+\nu_m}.
\]

Fix a colored multiset permutation
\(\pi^{\mathbf{c}}\).
A barred permutation on \(\pi^{\mathbf{c}}\) can be obtained by first
inserting one bar in every ascent gap and then inserting an arbitrary
number of additional bars in each of the \(m+1\) gaps. Therefore,
\[
\sum_{\theta\in\mathcal{B}(\pi^{\mathbf{c}})}
\operatorname{wt}(\theta)
=
x^{\operatorname{asc}(\pi^{\mathbf{c}})}
\left(\sum_{\ell\ge0}x^\ell\right)^{m+1}
=
\frac{x^{\operatorname{asc}(\pi^{\mathbf{c}})}}
     {(1-x)^{m+1}}.
\]
Summing this expression over all colored multiset permutations
\(\pi^{\mathbf{c}}\in
\mathfrak{S}_{\mathbf{m}}^{\mathbf{r}}\), we obtain
\begin{equation}\label{eq:first-barred-count}
\sum_{\theta\in\mathcal{B}}
\operatorname{wt}(\theta)
=
\frac{S_{\mathbf{m}}^{\mathbf{r}}(x)}
     {(1-x)^{m+1}}.
\end{equation}

We now evaluate the same total weight by grouping barred permutations
according to their total number of bars.
Let \(\mathcal{B}_n\subseteq\mathcal{B}\) denote the set of barred
permutations containing exactly \(n\) bars.
The \(n\) bars break the word into \(n+1\) segments, which may be
empty.
Since every ascent gap contains at least one bar, no segment can
contain an ascent.
Thus, all entries within each segment appear in weakly decreasing
order with respect to the fixed color order.

Each barred permutation in \(\mathcal{B}_n\) can be constructed by
first placing \(n\) bars in a line and then inserting the colored
copies of the letters into the \(n+1\) spaces determined by the bars.

We first consider the letter \(1\).
Since the sentinel \(1^1\) precedes the leftmost space and \(1^1\)
is the smallest element in the color order, only copies of \(1^1\)
may be inserted into this space.
This gives one possible choice.
For each of the remaining \(n\) spaces, the letter \(1\) may carry
any of the \(r_1\) available colors, giving \(r_1n\) further choices.
Thus, the \(m_1\) identical copies of \(1\) are chosen from
\(r_1n+1\) options, with repetition allowed.
The number of possible choices is therefore
\[
\binom{(r_1n+1)+m_1-1}{m_1}
=
\binom{r_1n+m_1}{m_1}.
\]

Now let \(i\in\{2,\ldots,d\}\).
The leftmost space is not available for copies of \(i\).
For each of the remaining \(n\) spaces, the letter \(i\) may carry
any of its \(r_i\) available colors.
Hence, the \(m_i\) identical copies of \(i\) are chosen from
\(r_in\) options, with repetition allowed.
The number of possible choices is
\[
\binom{r_in+m_i-1}{m_i}.
\]

The choices for distinct letters are mutually independent.
Moreover, once these choices are made, the weakly decreasing order
within each segment determines the barred permutation uniquely.
Therefore,
\[
\left|\mathcal{B}_n\right|
=
\binom{r_1n+m_1}{m_1}
\prod_{i=2}^{d}
\binom{r_in+m_i-1}{m_i}.
\]
Since every barred permutation in \(\mathcal{B}_n\) has weight
\(x^n\), summing over all nonnegative integers \(n\) gives
\begin{equation}\label{eq:second-barred-count}
\sum_{\theta\in\mathcal{B}}
\operatorname{wt}(\theta)
=
\sum_{n\ge0}
\binom{r_1n+m_1}{m_1}
\prod_{i=2}^{d}
\binom{r_in+m_i-1}{m_i}
x^n.
\end{equation}

We equate the two expressions
\eqref{eq:first-barred-count} and
\eqref{eq:second-barred-count} and the proof follows.
\end{proof}

\begin{proof}[Proof of Theorem~\ref{Sgeo}.]
Combining the definition of $Q_{\mathbf{m}}^{\mathbf{r}}$ with the formulas for the Ehrhart polynomial  of closed and half-open simplices,
we get
\[
\mathcal{L}(Q_{\mathbf{m}}^{\mathbf{r}};x)
= \mathcal{L}(r_1\Delta_{m_1};x)\prod_{j=2}^d\mathcal{L}(r_j\widetilde{\Delta}_{m_j};x)
= \binom{r_1 x+m_1}{m_1}\prod_{j=2}^d\binom{r_j x+m_j-1}{m_j}.
\]
Substituting into the generating function identity \eqref{eq:Sgf} gives
\[
\frac{S_{\mathbf{m}}^{\mathbf{r}}(x)}{(1-x)^{m+1}}
= \sum_{n\ge 0}\mathcal{L}(Q_{\mathbf{m}}^{\mathbf{r}};n)x^n.
\]
The conclusion now follows from the definition of the $h^*$-polynomial for half-open polytopes.
\end{proof}
\begin{proof}[Proof of Theorem~\ref{thm:Sself}.]
By Theorems~\ref{thm:Sgf} and~\ref{Sgeo},
\(Q_{\mathbf{m}}^{\mathbf{r}}\) has the Ehrhart polynomial
\[
\mathcal{L}(Q_{\mathbf{m}}^{\mathbf{r}};x)
= \binom{r_1x+m_1}{m_1}
  \prod_{j=2}^d\binom{r_jx+m_j-1}{m_j}.
\]

We first show that
\(\mathcal{L}(Q_{\mathbf{m}}^{\mathbf{r}};x)\)
expands nonnegatively in the magic basis. Indeed,
\begin{align*}
\mathcal{L}(Q_{\mathbf{m}}^{\mathbf{r}};x)
&= \binom{r_1x+m_1}{m_1}
   \prod_{j=2}^d\binom{r_jx+m_j-1}{m_j} \\
&= \prod_{k=0}^{m_1-1}
   \left(\frac{r_1}{m_1-k}x+1\right)
   \prod_{j=2}^{d}
   \left[
   \frac{r_j}{m_j}x
   \prod_{k=1}^{m_j-1}
   \left(\frac{r_j}{m_j-k}x+1\right)
   \right] \\
&= \prod_{k=0}^{m_1-1}
   \left[
   \left(\frac{r_1}{m_1-k}-1\right)x+(x+1)
   \right] \\
&\qquad {}\times
   \prod_{j=2}^{d}
   \left[
   \frac{r_j}{m_j}x
   \prod_{k=1}^{m_j-1}
   \left(
   \left(\frac{r_j}{m_j-k}-1\right)x+(x+1)
   \right)
   \right].
\end{align*}
If \(r_j\ge m_j+1\) for all \(j\) (or even \(r_j\ge m_j\)),
then \(\mathcal{L}(Q_{\mathbf{m}}^{\mathbf{r}};x)\) is magic positive.
Hence,
\[
\mathcal{E}
\bigl(
\mathcal{L}(Q_{\mathbf{m}}^{\mathbf{r}};x)
\bigr)
\in\mathcal{E}_m.
\]
By Lemma~\ref{BS21}, it therefore suffices to find a polynomial
\(p\in\mathcal{A}_m\) such that
\(
p\preceq
\mathcal{E}
\bigl(
\mathcal{L}(Q_{\mathbf{m}}^{\mathbf{r}};x)
\bigr).
\)
By~\cite[Corollary~3.9 and the discussion preceding Lemma~3.10]{DHS26},
\(
\mathcal{E}
\bigl(
\mathcal{L}(P_{\mathbf{m}}^{\mathbf{m}+\mathbf{1}};x)
\bigr)
\in\mathcal{A}_m.
\)
By~\cite[Theorem~4.6]{Bra06}, if two polynomials $p$ and $q$
have roots
\(
\alpha_1\leq\alpha_2\leq\cdots\leq\alpha_m\)
and
\(\beta_1\leq\beta_2\leq\cdots\leq\beta_m,
\)
respectively, all lying in the interval $[-1,0]$, and
$\alpha_k\leq\beta_k$ for all $k\in[m]$, then
$\mathcal{E}(p)\preceq\mathcal{E}(q)$.

We have
\[
\mathcal{L}
\bigl(
P_{\mathbf{m}}^{\mathbf{m}+\mathbf{1}};x
\bigr)
=
\prod_{j=1}^{d}
\binom{(m_j+1)x+m_j}{m_j}.
\]
Thus, its roots are
\[
\alpha_{jk}
=
-\frac{k}{m_j+1},
\qquad
j=1,\ldots,d,\quad k=1,\ldots,m_j.
\]
On the other hand,
the roots of
\(\mathcal{L}(Q_{\mathbf{m}}^{\mathbf{r}};x)\)
are given by
\[
\beta_{jk}
=
\begin{cases}
-\dfrac{k}{r_1},
& j=1,\quad k=1,\ldots,m_1,\\[6pt]
-\dfrac{k-1}{r_j},
& j\ge2,\quad k=1,\ldots,m_j.
\end{cases}
\]
All these roots lie in \([-1,0]\).

For \(j=1\), the assumption \(r_1\ge m_1+1\) gives
\[
\alpha_{1k}
=
-\frac{k}{m_1+1}
\le
-\frac{k}{r_1}
=
\beta_{1k}.
\]
For \(j\ge2\), since \(r_j\ge m_j+1\), we have
\[
kr_j
\ge k(m_j+1)
\ge (k-1)(m_j+1),
\]
and hence
\[
\alpha_{jk}
=
-\frac{k}{m_j+1}
\le
-\frac{k-1}{r_j}
=
\beta_{jk}.
\]
Therefore, the two multisets of roots can be paired so that
\(\alpha_{jk}\le\beta_{jk}\) for all \(j,k\). It follows from
\cite[Theorem~4.6]{Bra06} that
\[
\mathcal{E}
\bigl(
\mathcal{L}
(P_{\mathbf{m}}^{\mathbf{m}+\mathbf{1}};x)
\bigr)
\preceq
\mathcal{E}
\bigl(
\mathcal{L}
(Q_{\mathbf{m}}^{\mathbf{r}};x)
\bigr).
\]
Since
\(
\mathcal{E}
\bigl(
\mathcal{L}
(P_{\mathbf{m}}^{\mathbf{m}+\mathbf{1}};x)
\bigr)
\in\mathcal{A}_m
\)
and
\(
\mathcal{E}
\bigl(
\mathcal{L}
(Q_{\mathbf{m}}^{\mathbf{r}};x)
\bigr)
\in\mathcal{E}_m,
\)
Lemma~\ref{BS21} yields
\[
\mathcal{E}
\bigl(
\mathcal{L}
(Q_{\mathbf{m}}^{\mathbf{r}};x)
\bigr)
\in\mathcal{A}_m.
\]

Recall that the \emph{\(h\)-polynomial}
\(h(x)\in\mathbb{R}[x]\) associated with a polynomial
\(i(x)\in\mathbb{R}[x]\) of degree \(m\) is defined by
\[
\sum_{n\ge0}i(n)x^n
=
\frac{h(x)}{(1-x)^{m+1}},
\]
and its \emph{\(f\)-polynomial} is given by
\[
f(h;x)
=
(1+x)^m h\left(\frac{x}{1+x}\right).
\]
By~\cite[Lemma~2.7]{BS21},
\(
\mathcal{E}(i)(x)=f(h;x).
\)
Applying this identity to~\eqref{eq:Sgf}, with
\(
i(n)=\mathcal{L}(Q_{\mathbf{m}}^{\mathbf{r}};n)
\) and
\(h(x)=S_{\mathbf{m}}^{\mathbf{r}}(x),
\)
gives
\[
f(S_{\mathbf{m}}^{\mathbf{r}};x)
=
\mathcal{E}
\bigl(
\mathcal{L}
(Q_{\mathbf{m}}^{\mathbf{r}};x)
\bigr).
\]
Since
\(
\mathcal{E}
\bigl(
\mathcal{L}
(Q_{\mathbf{m}}^{\mathbf{r}};x)
\bigr)
\in\mathcal{A}_m,
\)
we have
\[
\mathcal{R}_m
\bigl(
f(S_{\mathbf{m}}^{\mathbf{r}};x)
\bigr)
\preceq
f(S_{\mathbf{m}}^{\mathbf{r}};x).
\]
By~\cite[Lemma~2.2]{BS21},
\[
\mathcal{R}_m
\bigl(
f(S_{\mathbf{m}}^{\mathbf{r}};x)
\bigr)
=
f
\bigl(
\mathcal{I}_m(S_{\mathbf{m}}^{\mathbf{r}});x
\bigr).
\]
It follows that
\[
f
\bigl(
\mathcal{I}_m(S_{\mathbf{m}}^{\mathbf{r}});x
\bigr)
\preceq
f(S_{\mathbf{m}}^{\mathbf{r}};x).
\]

Finally, the transformation
\[
h(x)\longmapsto
f(h;x)
=
(1+x)^m h\left(\frac{x}{1+x}\right)
\]
is induced on the roots by the fractional-linear map:
\(
u\longmapsto\frac{u}{1-u},
\)
which is strictly increasing on \((-\infty,0]\).
Consequently, this transformation preserves and reflects
interlacing. We conclude that
\[
\mathcal{I}_m(S_{\mathbf{m}}^{\mathbf{r}})
\preceq
S_{\mathbf{m}}^{\mathbf{r}},
\]
as desired.
\end{proof}
\begin{cor}\label{Sprops}
If \(r_j\ge m_j\) for all $j\in [d]$,
then \(S_{\mathbf{m}}^{\mathbf{r}}\) is real-rooted,
log-concave,
unimodal.
Moreover,
if \(r_j\ge m_j+1\) for all $j\in [d]$,
then \(S_{\mathbf{m}}^{\mathbf{r}}\) is alternatingly increasing and bi-\(\gamma\)-positive.
\end{cor}
\begin{proof}
From the proof of Theorem~\ref{thm:Sself},
the Ehrhart polynomial $\mathcal{L}(Q_{\mathbf{m}}^{\mathbf{r}};x)$ is magic positive when  \(r_j\ge m_j\) for all $j$.
A theorem of Br\"and\'en \cite[Theorem 4.2]{Bra06} implies in this situation that the $h^*$-polynomials $S_{\mathbf{m}}^{\mathbf{r}}$ have only real roots,
which implies log-concavity and unimodality.
When \(r_j\ge m_j+1\) holds for all $j$,
we apply the discussion following Lemma 3.1 in \cite{DBZ24}.
If a polynomial $p$ of degree at most $m$ with nonnegative coefficients satisfies \(\mathcal{I}_m(p)\preceq p\),
then its \(\mathcal{I}_m\)-decomposition has nonnegative coefficients and interlacing roots.
Applying this to \(p=S_{\mathbf{m}}^{\mathbf{r}}\) and using Theorem~\ref{thm:Sself},
we obtain the alternatingly increasing property and bi-\(\gamma\)-positivity.
\end{proof}
\section{Real-Rootedness}\label{4}
\hspace*{\parindent}
This section proves Theorems \ref{rS} and \ref{rA}.
The following two statements provide recurrences for $S_{\mathbf{m}}^{\mathbf{r}}$ and $A_{\mathbf{m}}^{\mathbf{r}}$;
they are crucial for these two theorems.

\begin{prop}\label{reS}
Increasing the $j$-th multiplicity by one gives the following recurrence for $S_{\mathbf{m}}^{\mathbf{r}}$.

If $j=1$, then
\begin{equation}\label{S1}
S_{\mathbf{m}_{+1}}^{\mathbf{r}}(x)
= \left(1 + \left(\frac{r_1(m+1)}{m_1+1}-1\right)x\right)S_{\mathbf{m}}^{\mathbf{r}}(x)
+ \frac{r_1}{m_1+1}x(1-x)(S_{\mathbf{m}}^{\mathbf{r}}(x))^{\prime}.
\end{equation}

If $j\ge 2$, then
\begin{equation}\label{S2}
S_{\mathbf{m}_{+j}}^{\mathbf{r}}(x)
= \left(\frac{m_j}{m_j+1} + \frac{r_j(m+1)-m_j}{m_j+1}x\right)S_{\mathbf{m}}^{\mathbf{r}}(x)
+ \frac{r_j}{m_j+1}x(1-x)(S_{\mathbf{m}}^{\mathbf{r}}(x))^{\prime}.
\end{equation}
\end{prop}
\begin{proof}
For $j=1$, define the generating functions
\[
F(x)=\frac{S_{\mathbf{m}_{+1}}^{\mathbf{r}}(x)}{(1-x)^{m+2}},\qquad
G(x)=\frac{S_{\mathbf{m}}^{\mathbf{r}}(x)}{(1-x)^{m+1}}.
\]
By \eqref{eq:Sgf},
\[
F(x)=\sum_{n \ge 0} \binom{r_1 n + m_1+1}{m_1+1} \prod_{i=2}^d \binom{r_i n + m_i - 1}{m_i} x^n,
\]
\[
G(x)=\sum_{n \ge 0} \binom{r_1 n + m_1}{m_1} \prod_{i=2}^d \binom{r_i n + m_i - 1}{m_i} x^n.
\]
Only the first binomial factor is updated from $\binom{r_1 n+m_1}{m_1}$ to $\binom{r_1 n+m_1+1}{m_1+1}$.
Using the identity
\[
\binom{r_1 n+m_1+1}{m_1+1}
= \left(1+\frac{r_1}{m_1+1}n\right)\binom{r_1 n+m_1}{m_1},
\]
we obtain
\[
F(x)=G(x)+\frac{r_1}{m_1+1}x\,G'(x).
\]
We substitute the expressions of $F(x)$ and $G(x)$ in terms of $S_{\mathbf{m}_{+1}}^{\mathbf{r}}$ and $S_{\mathbf{m}}^{\mathbf{r}}$, multiply both sides by $(1-x)^{m+2}$, expand the derivative term, and simplify to recover \eqref{S1}.

For $j\ge 2$, let
\[
F_1(x)=\frac{S_{\mathbf{m}_{+j}}^\mathbf{r}(x)}{(1-x)^{m+2}}
=\sum_{n \ge 0} \binom{r_1 n + m_1}{m_1}
\left(\prod_{\substack{i=2\\i\ne j}}^d \binom{r_i n + m_i - 1}{m_i}\right)
\binom{r_j n + m_j}{m_j+1} x^n.
\]
Here,
only the $j$-th factor changes from $\binom{r_j n+m_j-1}{m_j}$ to $\binom{r_j n+m_j}{m_j+1}$.
The identity
\[
\binom{r_j n+m_j}{m_j+1}
= \left(\frac{m_j}{m_j+1}+\frac{r_j}{m_j+1}n\right)\binom{r_j n+m_j-1}{m_j}
\]
gives
\[
F_1(x)=\frac{m_j}{m_j+1}G(x)+\frac{r_j}{m_j+1}x\,G'(x).
\]
Substituting and simplifying yields (\ref{S2}).
\end{proof}
\begin{proof}[Proof of Theorem~\ref{rS}]
By Corollary~\ref{Sprops}, $S_{\mathbf{1}}^{\mathbf{r}}$ is real-rooted for every $\mathbf{r}\in\mathbb{Z}_{>0}^d$.
Given the recurrence from Proposition~\ref{reS} and Theorem~\ref{LW07}, the result follows by induction.
\end{proof}
Similarly,
we can derive a recurrence relation for $A_{\mathbf{m}}^{\mathbf{r}}$.
\begin{prop}\label{reA}
Increasing the $j$-th multiplicity by one gives
\begin{equation}\label{A2}
A_{\mathbf{m}_{+j}}^{\mathbf{r}}(x)
= \left(1+\left(\frac{r_j(m+1)}{m_j+1}-1\right)x\right)A_{\mathbf{m}}^{\mathbf{r}}(x)
+\frac{r_j}{m_j+1}x(1-x)(A_{\mathbf{m}}^{\mathbf{r}}(x))^{\prime}.
\end{equation}
\end{prop}
\begin{proof}
Let
\[
I(x)=\frac{A_{\mathbf{m}_{+j}}^{\mathbf{r}}(x)}{(1-x)^{m+2}},\]
\[H(x)=\frac{A_{\mathbf{m}}^{\mathbf{r}}(x)}{(1-x)^{m+1}}.
\]
By \eqref{A1},
\[
I(x)=\sum_{n\ge 0}
\left(\prod_{\substack{i=1\\i\ne j}}^d \binom{r_i n+m_i}{m_i}\right)
\binom{r_j n+m_j+1}{m_j+1}
x^n.
\]
When increasing $m_j$ by one,
only the $j$-th factor changes and becomes $\binom{r_j n+m_j+1}{m_j+1}$ instead of $\binom{r_j n+m_j}{m_j}$. Using
\[
\binom{r_j n+m_j+1}{m_j+1}
= \left(1+\frac{r_j}{m_j+1}n\right)\binom{r_j n+m_j}{m_j},
\]
we get
\[
I(x)=H(x)+\frac{r_j}{m_j+1}xH'(x).
\]
Substituting and simplifying yields (\ref{A2}).
\end{proof}

\begin{remark}\label{baseA}
The proof of \cite[Theorem 3.11]{DHS26} implies the weaker statement that $A_{\mathbf{m}}^{\mathbf{r}}$ is real-rooted if $r_j\geq m_j$ for every $j$.
\end{remark}
\begin{proof}[Proof of Theorem~\ref{rA}]
By Remark~\ref{baseA}, $A_{\mathbf{1}}^{\mathbf{r}}$ is real-rooted for every positive integer vector $\mathbf{r}\in\mathbb{Z}_{>0}^d$.
Using the recurrence in Proposition~\ref{reA} and Theorem~\ref{LW07},
the proof follows by induction.
\end{proof}
\section{Applications}\label{5}
\hspace*{\parindent}
The set $\mathfrak{S}_{\mathbf{m}}^{\mathbf{r}}$ of colored multiset permutations was previously defined only for positive multiplicity vectors $\mathbf{m}\in\mathbb{Z}_{>0}^d$. In this section, we extend its definition to all nonnegative integer vectors $\mathbf{m}=\left(m_1,\dots,m_d\right)\in\mathbb{N}^d$ with $m=m_1+m_2+\dots+m_d$ and denote this extended collection by $\widetilde{\mathfrak{S}}_{\mathbf{m}}^{\mathbf{r}}$.
Each element of
$\widetilde{\mathfrak{S}}_{\mathbf{m}}^{\mathbf{r}}$
is a colored word
\(
w^{\mathbf{c}}
=
w_1^{c_1}w_2^{c_2}\cdots w_m^{c_m},
\)
where $w_1w_2\cdots w_m\in[d]^m$ and each
$i\in[d]$ appears exactly $m_i$ times.
The coloring rule $c_{i}\in [r_{w_i}]$ and the sentinel element $w_0^{c_0}=1^1$ are defined in the same way as in Section 1.
For any index $j\in[0,m-1]$, we say that $j$ is an \emph{ascent} of $w^{\mathbf{c}}$ if $w_j^{c_j}<w_{j+1}^{c_{j+1}}$, and we define $\operatorname{asc}(w^{\mathbf{c}})$ as the number of ascents.

Let $\mathcal{G}\subseteq\mathbb{N}^d$ be an arbitrary set of vectors with coordinate sum equal to $m$.
We define the set of extended colored words
\[
\mathcal{W}_{\mathcal{G}}^{\mathbf{r}} = \bigcup_{\mathbf{m}\in\mathcal{G}} \widetilde{\mathfrak{S}}_{\mathbf{m}}^{\mathbf{r}}.
\]

The following statement is a consequence of  Theorem~\ref{thm:Sgf}.
\begin{cor}\label{cor:wr}
Let $\mathcal{G}\subseteq\mathbb{N}^d$ be as above and let
$\mathbf{r}=(r_1,r_2,\dots,r_d)\in\mathbb{Z}_{>0}^d$ be a positive integer vector.
Then,
\[
\frac{\displaystyle\sum_{w^{\mathbf{c}} \in \mathcal{W}_{\mathcal{G}}^{\mathbf{r}}} x^{\operatorname{asc}(w^{\mathbf{c}})}}{(1-x)^{m+1}}
=
\sum_{n\ge0}
\left(
\sum_{\mathbf{m}\in\mathcal{G}}
\binom{r_1 n+m_1}{m_1}
\prod_{i=2}^d
\binom{r_in+m_i-1}{m_i}
\right)x^n.
\]
\end{cor}

\begin{proof}
For every fixed vector $\mathbf{m}\in\mathcal{G}$,
the proof of Theorem~\ref{thm:Sgf} remains valid for the extended
colored word set
$\widetilde{\mathfrak{S}}_{\mathbf{m}}^{\mathbf{r}}$.
Summing the resulting generating function identities over all
$\mathbf{m}\in\mathcal{G}$ leads to the desired equality.
\end{proof}

As applications,
we provide combinatorial interpretations of the $h^*$-polynomials of several families of lattice polytopes.
\subsection{Pitman--Stanley polytope}
\hspace*{\parindent}
For $\mathbf a=(a_1,\ldots,a_d)\in{\mathbb N}^d$,
the {\it Pitman--Stanley polytope} is defined by
\[
PS_d(\mathbf a)
=
\left\{
(x_1,\ldots,x_d)\in\mathbb R_{\ge0}^d:
\sum_{j=1}^i x_j
\le
\sum_{j=1}^i a_j,
\quad
1\le i\le d
\right\}.
\]

Pitman and Stanley \cite[equation (33)]{PS02} proved the explicit formula for the Ehrhart polynomial
\begin{equation}\label{equ:ps}
\mathcal{L}(PS_d(\mathbf a),n)
=
\sum_{\mathbf m\in I_d}
\binom{a_1n+m_1}{m_1}
\prod_{i=2}^d
\binom{a_in+m_i-1}{m_i},
\end{equation}
where
\[
I_d
=
\left\{
(m_1,\ldots,m_d)\in\mathbb N^d:
\sum_{i=1}^d m_i=d,\;
\sum_{i=1}^k m_i\ge k,\quad
1\le k\le d
\right\}.
\]

Applying Corollary~\ref{cor:wr} to~\eqref{equ:ps} yields the following.
\begin{cor}\label{cor:PS}
Suppose $\mathbf r=(r_1,\dots,r_d)\in\mathbb Z_{>0}^d$.
We have
\[
h^*\big(PS_d(\mathbf r);x\big)
=
\sum_{w^{\mathbf{c}} \in \mathcal{W}_{I_d}^{\mathbf{r}}} x^{\operatorname{asc}(w^{\mathbf{c}})},
\]
where $\mathcal{W}_{I_d}^{\mathbf{r}}$ consists of all extended colored words $w^{\mathbf{c}}$ whose underlying word $w\in [d]^d$ has at least $k$ entries less than or equal to $k$ for every $k\in [d]$.
Equivalently, the nondecreasing rearrangement $(u_1,u_2,\dots,u_d)$ of $w$ satisfies $u_k\le k$ for every $k\in [d]$.
\end{cor}

\begin{proof}
Substituting $\mathcal{G}=I_d$ into Corollary~\ref{cor:wr}
and combining the resulting identity with equation~\eqref{equ:ps},
we get
\[
\frac{\displaystyle\sum_{w^{\mathbf{c}} \in \mathcal{W}_{I_d}^{\mathbf{r}}} x^{\operatorname{asc}(w^{\mathbf{c}})}}
     {(1-x)^{d+1}}
=
\sum_{n\ge0}
\mathcal{L}(PS_d(\mathbf r),n)x^n.
\]
The result follows from this identity and the definition of the $h^*$-polynomial.
\end{proof}

Avila, Ferroni and Morales~\cite[Theorem~1.3]{AFM26}
recently gave, more generally for
$\mathbf a\in\mathbb N^d$, a combinatorial interpretation of
$h^*(PS_d(\mathbf a);x)$ as the modified-ascent enumerator of
$\mathbf a$-parking functions.
For $\mathbf a=\mathbf r\in\mathbb Z_{>0}^d$,
Corollary~\ref{cor:PS} therefore provides an alternative combinatorial
interpretation of the same $h^*$-polynomial in terms of colored words.

\subsection{Composition polytopes}
\hspace*{\parindent}
Let $\sigma = (b_1,b_2,\dots,b_k) \models d$ be a composition of a positive integer $d$
and consider the partial sums $s_i = b_1 + b_2 + \dots + b_i$ for $i\in [k]$.
The {\it composition polytope} $\mathcal{Q}_\sigma \subset \mathbb{R}^d$ is defined by the inequalities
\[
\mathcal{Q}_\sigma = \left\{ (x_1,x_2,\dots,x_d) \,\bigg|\, x_j\ge 0,\ \forall j\in[d],\ \sum_{j=1}^{s_i} x_j \le s_i,\ \forall\, i\in [k] \right\}.
\]
Let
\[
K_\sigma = \left\{ \mathbf{m} = (m_1,\dots,m_d) \in \mathbb{N}^d \,\bigg|\, \sum_{j=1}^d m_j = d,\ \sum_{j=1}^{s_i} m_j \ge s_i,\ \forall\, i\in [k] \right\}.
\]
From \cite[Equation (13)]{Ath26}, the Ehrhart polynomial of $\mathcal{Q}_\sigma$ admits the explicit formula
\begin{equation}\label{equ:comp}
\mathcal{L}(\mathcal{Q}_\sigma;n) = \sum_{\mathbf{m} \in K_\sigma}
\binom{n+m_1}{m_1} \prod_{i=2}^d \binom{n+m_i-1}{m_i}.
\end{equation}

Given a word $w = (w_1,w_2,\dots,w_d) \in [d]^d$,
we denote by $\operatorname{asc}(w)$ the number of indices (called ascents) $i\in [d]$ such that $w_{i-1} < w_i$, where $w_0 := 1$.

\begin{cor}[{[\citen{Ath26}]}, Proposition 4.2]\label{cor:comp}
We have
\[
h^*\big(\mathcal{Q}_\sigma;x\big) = \sum_{w \in \mathcal{W}_{\sigma}} x^{\operatorname{asc}(w)},
\]
where $\mathcal{W}_{\sigma}$ consists of all words $w\in [d]^d$ which have at least $s_i$ entries less than or equal to $s_i$ for every $i\in [k]$.
Equivalently,
the unique nondecreasing rearrangement $(u_1,u_2,\dots,u_d)$ of $w$ satisfies $u_{s_i} \le s_i$ for every $i\in [k]$.
\end{cor}

\begin{proof}
Substituting $\mathcal{G}=K_\sigma$ and the trivial color vector
$\mathbf{r}=\mathbf{1}$ into Corollary~\ref{cor:wr}
and combining the resulting identity with formula~\eqref{equ:comp},
we obtain
\[
\frac{\displaystyle\sum_{w^{\mathbf{c}} \in \mathcal{W}_{K_\sigma}^{\mathbf{1}}} x^{\operatorname{asc}(w^{\mathbf{c}})}}{(1-x)^{d+1}}
= \sum_{n \ge 0}\mathcal{L}(\mathcal{Q}_\sigma;n) x^n.
\]
Words equipped with trivial coloring $\mathbf{r}=\mathbf{1}$ correspond bijectively to uncolored words in $\mathcal{W}_\sigma$.
Hence,
\[
\sum_{w^{\mathbf{c}}\in \mathcal{W}_{K_\sigma}^{\mathbf{1}}}x^{\operatorname{asc}(w^{\mathbf{c}})}=\sum_{w\in \mathcal{W}_\sigma}x^{\operatorname{asc}(w)}.
\]
The claim now follows from the definition of the $h^*$-polynomial of a lattice polytope.
\end{proof}

\subsection{Reflexive lattice polytopes defined from preorders}
\hspace*{\parindent}
Let $E$ be a $d$-element set.
A \emph{preorder} $\tau$ on $E$ is a reflexive and transitive relation $\leq_\tau$. An \emph{order ideal} of $\tau$ is a subset $\mathcal{I}\subseteq E$ such that $b\in\mathcal{I}$ and $a\leq_\tau b$ imply $a\in\mathcal{I}$. The \emph{preorder polytope} $\mathcal{Q}_\tau\subseteq\mathbb{R}^E$ is defined by
\[
\mathcal{Q}_\tau=
\left\{
\mathbf{x}\in\mathbb{R}^E
\,\bigg|\,
x_e\geq0\ \text{for all }e\in E,\quad
\sum_{e\in\mathcal{I}}x_e\leq|\mathcal{I}|
\ \text{for every order ideal }\mathcal{I}\text{ of }\tau
\right\}.
\]
Let $\Delta_E$ denote the standard simplex in $\mathbb{R}^E$ and let $\mathbf{1}=\sum_{e\in E}\mathbf{e}_e$. Following Section~6 of~\cite{AC26}, let
\[
R_\tau=\mathcal{Q}_\tau+\Delta_E-\mathbf{1}.
\]
By Equations~(23) and~(24) in~\cite{AC26}, the polytope $R_\tau$ is characterized by the inequalities $x_e\geq-1$ for all $e\in E$ and
$\sum_{e\in\mathcal{I}}x_e\leq1$ for every order ideal $\mathcal{I}$ of $\tau$. Moreover, $R_\tau$ is a reflexive lattice polytope by~\cite[Proposition~6.1]{AC26} and hence its $h^*$-polynomial is palindromic of degree $d$.

Let $1\notin E$ be an auxiliary symbol. Since $R_\tau$ is a lattice translation of $\mathcal{Q}_\tau+\Delta_E$, Equation~(25) and Lemma~4.1 in~\cite{AC26} give
\begin{equation}\label{ehr:Rt}
\mathcal{L}(R_\tau;n)
=
\sum_{\mathbf{m}\in\mathcal{G}_\tau}
\binom{n+m_1}{m_1}
\prod_{e\in E}\binom{n+m_e-1}{m_e},
\end{equation}
where
\[
\mathcal{G}_\tau=
\left\{
(m_1,(m_e)_{e\in E})\in\mathbb{N}^{E\sqcup\{1\}}
\,\bigg|\,
m_1+\sum_{e\in E}m_e=d,\
m_1+\sum_{e\in\mathcal{I}}m_e\geq|\mathcal{I}|
\ \text{for every order ideal }\mathcal{I}\text{ of }\tau
\right\}.
\]
Indeed, Lemma~4.1 gives
\[
\sum_{e\in\mathcal{J}}m_e\leq|\mathcal{J}|
\]
for every order filter $\mathcal{J}$ of $\tau$.
Taking $\mathcal{J}=E\setminus\mathcal{I}$ and using $m_1+\sum_{e\in E}m_e=d$ gives the stated characterization of $\mathcal{G}_\tau$.

Fix an arbitrary total order $<$ on $E$ and extend it to $E\sqcup\{1\}$ by declaring $1< e$ for every $e\in E$.
Given a word $w=w_1w_2\cdots w_d\in(E\sqcup\{1\})^d$,
set $w_0=1$ and let $\operatorname{asc}(w)$ be the number of indices $i\in[0,d-1]$ such that $w_i< w_{i+1}$.

\begin{cor}\label{cor:Q11}
For every preorder $\tau$ on $E$ we have
\[
h^*(R_\tau;x)
=
\sum_{w\in\mathcal{W}_\tau}x^{\operatorname{asc}(w)},
\]
where $\mathcal{W}_\tau$ is the set of such words  $w\in (E\sqcup\{1\})^d$ satisfying
\(
\#\{j\in[d]:w_j\in\mathcal{I}\cup\{1\}\}
\geq|\mathcal{I}|
\)
for every order ideal $\mathcal{I}$ of $\tau$.
\end{cor}

\begin{proof}
For $w\in(E\sqcup\{1\})^d$, let $m_1$ be the number of appearances of $1$ in $w$ and let $m_e$ be the number of appearances of $e$ for every $e\in E$.
Then,
$w\in\mathcal{W}_\tau$ if and only if $(m_1,(m_e)_{e\in E})\in\mathcal{G}_\tau$.
Relabeling the totally ordered alphabet $E\sqcup\{1\}$ as $[d+1]$ and applying Corollary~\ref{cor:wr} with $\mathcal{G}=\mathcal{G}_\tau$ and the trivial color vector $\mathbf{r}=\mathbf{1}$, we obtain
\[
\frac{\displaystyle\sum_{w\in\mathcal{W}_\tau}
x^{\operatorname{asc}(w)}}{(1-x)^{d+1}}
=
\sum_{n\geq0}
\left(
\sum_{\mathbf{m}\in\mathcal{G}_\tau}
\binom{n+m_1}{m_1}
\prod_{e\in E}\binom{n+m_e-1}{m_e}
\right)x^n.
\]
By~\eqref{ehr:Rt}, the right-hand side is equal to
\[
\sum_{n\geq0}\mathcal{L}(R_\tau;n)x^n
=
\frac{h^*(R_\tau;x)}{(1-x)^{d+1}}
\]
and the proof follows.
\end{proof}

\medskip
\noindent
\textbf{Acknowledgements}. The author acknowledges support from the China Scholarship Council
(no.~202508370093) for her visit to the Department of Mathematics of
the National and Kapodistrian University of Athens during the academic
year 2025--26.
The author is especially grateful to Christos A.~Athanasiadis
for suggesting the problems studied in this paper,
for many helpful discussions,
and for his valuable comments and suggestions on earlier versions of this paper.
The author is also grateful to Danai Deligeorgaki for her helpful comments and suggestions.


\begin{thebibliography}{99}

\bibitem{Ath18}
C. A. Athanasiadis,
Gamma-positivity in combinatorics and geometry,
S\'em. Lothar. Combin.
77 (2018), Art.~B77i, 64 pp.

\bibitem{Ath26}
C. A. Athanasiadis,
Lattice point enumeration of polytopes associated to integer compositions,
Ann. Comb.,
published online, 2026.
doi:10.1007/s00026-026-00812-2.

\bibitem{AC26}
C. A. Athanasiadis,
F. Chapoton,
Polytopes and posets associated to preorders,
arXiv:2605.26916.

\bibitem{AFM26}
N. Avila, L. Ferroni and A. H. Morales,
Luck and magic for Pitman--Stanley polytopes,
extended abstract, Proceedings of FPSAC 2026, 12 pp.

\bibitem{Bra06}
P. Br\"and\'en,
On linear transformations preserving the P\'olya frequency property,
Trans. Amer. Math. Soc.
358(8) (2006), 3697--3716.

\bibitem{Bra15}
P. Br\"and\'en,
Unimodality, log-concavity, real-rootedness and beyond,
Handbook of Enumerative Combinatorics,
CRC Press, Boca Raton, FL, 2015, pp.~437--483.

\bibitem{BR07}
M. Beck,
S. Robins,
Computing the Continuous Discretely:
Integer-Point Enumeration in Polyhedra,
Undergrad. Texts Math.,
Springer, New York, 2007.

\bibitem{BS21}
P. Br\"and\'en,
L. Solus,
Symmetric decompositions and real-rootedness,
Int. Math. Res. Not. IMRN
2021(10) (2021), 7764--7798.

\bibitem{Com74}
L. Comtet,
Advanced Combinatorics,
Reidel, Dordrecht, 1974.

\bibitem{DHS26}
D. Deligeorgaki,
B. Han,
L. Solus,
Colored multiset Eulerian polynomials,
Combin. Theory
6(1) (2026), \#10.

\bibitem{DBZ24}
M.-J. Ding,
B.-X. Zhu,
Real stable polynomials and the alternatingly increasing property,
European J. Combin.
120 (2024), 103944.

\bibitem{FS70}
D. Foata,
M.-P. Sch\"utzenberger,
Th\'eorie g\'eometrique des polyn\^{o}mes eul\'eriens,
Lecture Notes in Math.,
Vol.~138,
Springer, Berlin--New York, 1970.

\bibitem{Fro10}
G. Frobenius,
\"Uber die Bernoullischen Zahlen und die Eulerschen Polynome,
Sitzungsber. Preuss. Akad. Wiss. Berlin
(1910), 809--847.

\bibitem{Gal05}
S. R. Gal,
Real root conjecture fails for five- and higher-dimensional spheres,
Discrete Comput. Geom.
34 (2005), 269--284.

\bibitem{GS78}
I. M. Gessel and R. P. Stanley,
Stirling polynomials,
J. Combin. Theory Ser. A 24 (1978), 24--33.

\bibitem{Li12}
N. Li,
Ehrhart $h^*$-vectors of hypersimplices,
Discrete Comput. Geom.
48(4) (2012), 847--878.
\bibitem{Lin15}

Z. Lin,
On the descent polynomial of signed multipermutations,
Proc. Amer. Math. Soc.
143(9) (2015), 3671--3685.

\bibitem{LW07}
L. L. Liu,
Y. Wang,
A unified approach to polynomial sequences with only real zeros,
Adv. Appl. Math.
38 (2007), 542--560.


\bibitem{Mac04}
P. A. MacMahon,
Combinatory Analysis,
Vol.~II,
Dover Publications, Mineola, NY, 2004.

\bibitem{Pet15}
T. K. Petersen,
Eulerian Numbers,
Birkh\"auser, Basel, 2015.

\bibitem{PS02}
J. Pitman,
R. P. Stanley,
A polytope related to empirical distributions, plane trees, parking functions, and the associahedron,
Discrete Comput. Geom.
27(4) (2002), 603--634.


\bibitem{Sim84}
R. Simion,
A multi-indexed Sturm sequence of polynomials and unimodality of certain combinatorial sequences,
J. Combin. Theory Ser. A
36(1) (1984), 15--22.

\bibitem{Sta89}
R. P. Stanley,
Log-concave and unimodal sequences in algebra, combinatorics, and geometry,
in: Graph Theory and Its Applications: East and West,
Ann. New York Acad. Sci.
576 (1989), 500--535.


\bibitem{Ste92}
E. Steingr\'{\i}msson,
Permutations statistics of indexed and poset permutations,
Thesis (Ph.D.)-Massachusetts Institute of Technology, ProQuest LLC, Ann Arbor, MI, 1992.

\bibitem{Ste94}
E. Steingr\'imsson,
Permutation statistics of indexed permutations,
European J. Combin.
15(2) (1994), 187--205.


\bibitem{Tie26}
E. Tielker,
Weighted Ehrhart series and a type-$B$ analogue of a formula of MacMahon,
Discrete Comput. Geom.
75 (2026), 205--240.

\end{thebibliography}
\end{document}